\documentclass[10pt]{article}
\usepackage{times}
\usepackage{amsmath,amsfonts,amstext,amssymb,amsbsy,amsopn,amsthm,eucal,slashed,amstext,enumerate,accents}
\usepackage{color}
\usepackage{txfonts}
\usepackage{dsfont}
\usepackage{hyperref}
\usepackage{graphicx}   
\usepackage[colorinlistoftodos]{todonotes}

\numberwithin{equation}{section}
  \theoremstyle{plain}
 \newtheorem{theorem}{Theorem}
 
  \newtheorem{theoremletter}{Theorem}

\newtheorem*{convention}{\textbf{Convention}}
 \newtheorem{lemma}{Lemma}
 \newtheorem{corollary}{Corollary}

 \newtheorem{remark}{Remark}

\theoremstyle{definition}

\def \no{\nonumber}

\newcommand{\Ric}{\mathrm{Ric}}

\newcommand{\diam}{{\rm diam}}

\newcommand{\R}{{ \mathrm R}}

\newcommand{\pa}{\partial}
\newcommand{\Rm}{\mathrm{Rm}}

\newcommand{\Sn}{\mathbb{S}^n}
\newcommand{\Sp}{\mathbb{S}}

\newcommand{\fR}{\mathfrak{R}}

\newcommand{\tr}{\mathrm{tr}}
\newcommand{\Lap}{\Delta}

\newcommand{\Vol}{\mathrm{Vol}}
\newcommand{\RR}{\mathbb{R}}
\newcommand{\Rn}{\mathbb{R}^n}

\newcommand{\ud}{\,\mathrm{d}}
\newcommand{\md}{\mathrm{d}}
\newcommand{\RP}{\mathbb {RP}}
\newcommand{\ve}{\varepsilon}

\allowdisplaybreaks

\title{A sharp mass-volume inequality on Einstein manifolds and ADM mass}
\author{Xuezhang Chen\thanks{$^\dag$$^\S$ X. Chen, J. Gan and N. Wu are partially supported by NSFC (No.11771204). Email: xuezhangchen@nju.edu.cn.},~ Jiaxue Gan\thanks{J. Gan's Email: 602025210004@smail.nju.edu.cn.},~Yalong Shi\thanks{Y. Shi is partially supported by NSFC (No.12371058). Email:shiyl@nju.edu.cn.}~ and Nan Wu\thanks{N. Wu's Email: nwu2@nju.edu.cn}\\
		{\small $^{\ast}$$^{\dag}$$^{\ddag}$$^\S$School of Mathematics \& IMS, Nanjing University, Nanjing 210093, P.R. China }}
\date{}

\begin{document}
\maketitle

\begin{abstract}
Let $(M^n,g), n\geq 4$ be a close Einstein manifold with positive scalar curvature and $G$ be the Green function of the conformal Laplacian. The ADM mass of the asymptotically flat metric $\hat g:=G^{4/(n-2)} g$ is an important geometric invariant. When either $n=5$ or  $n\geq 6$ with  Weyl vanishing condition, we prove a mass formula, generalizing Gursky-Malchiodi's formula in dimension $4$. An important step is an extension of Viaclovsky's Green function expansion in higher dimensions. Next we prove a sharp inequality related to the volume radius and the mass for Einstein manifolds. Finally we present two applications. First, we prove a mass gap theorem for Einstein manifolds. Second, for a volume collapsing family of normalized Einstein $n$-manifolds $(M_i,g_i)$ (satisfying the Weyl vanishing condition when $n\geq 6$), we prove that the mass sequence  $m(\widehat{g_i})$ blows up uniformly as $i\to\infty$. This in particular applies to
 Collins-Sz\'ekelyhidi's family of Sasaki-Einstein metrics $g_i$ on $\Sp^5$, and to the recent $4$-dimensional examples of T. Ozuch.

\medskip

\end{abstract}
\medskip

{\footnotesize \tableofcontents}

\section{Introduction}

Let $(N^n,g)$ be an asymptotically flat manifold of  order $\tau>0$, in the sense that there exists a compact set $K\subset N$ such that  $N\setminus K$ is diffeomorphic to $ \mathbb{R}^n \backslash B_R(0)$ for some $R>0$, and under asymptotic coordinate chart, $g-\delta \in O_{2}(|y|^{-\tau})$ as $|y| \to \infty$.   The ADM mass of $(N^n,g)$ was defined by Arnowitt-Deser-Misner \cite{ADM} in the 1960s in the context of general relativity: 
\begin{equation}\label{Def:ADM_mass}
m(g):=\frac{1}{|\Sp^{n-1}|}\lim_{t \to \infty} \int_{\Sp_t^{n-1}}g^{\beta \gamma}(\pa_\gamma g_{\alpha \beta}-\pa_\alpha g_{\beta \gamma})\nu^\alpha \ud\sigma
\end{equation}
if the limit exists.  In \cite{bartnik}, Bartnik proved that when the asymptotic order $\tau>\frac{n-2}{2}$ and the scalar curvature is integrable, the mass is independent of the choice of asymptotic coordinates and is thus a geometric invariant for $(N,g)$.

For a closed Riemannian manifold $(M^n,g),n\geq 3$ with positive Yamabe constant,  the conformal Laplacian $L_g=-\Delta_g +\frac{n-2}{4(n-1)}R_g$ is related to the ADM mass in the following way: Let $G$ denote the Green function of $L_g$ with pole at some $\mathsf{p} \in M$ such that $L_g G=(n-2)|\Sp^{n-1}|\delta_{\mathsf{p}}$ in the distribution sense, then $(M\setminus\{\mathsf{p} \},\hat g:=G^{4/(n-2)}g)$ is a scalar-flat complete manifold, and of asymptotic order $1+\lfloor \frac{n-2}{2}\rfloor$ when $n=3,4,5$, and order $n-2$ when $(M,g)$ is locally conformally flat manifold and $n\geq 6$,  so the mass is well-defined. This fact plays a decisive role in the resolution of Yamabe problem; see \cite{Lee-Parker} for a survey.  Stimulated by this fact,  in Chen-Shi \cite{Chen-Shi} and Chen-Gan-Shi \cite{CGS}, the authors established the Green function rigidity conjecture for the conformal Laplacian/the Paneitz operator by using the Positive Mass/Energy Theorem, respectively.

Another closely related topic is the mass of an Einstein $4$-manifold. Viaclovsky \cite{Viaclovsky} initiated the study of the Green function expansion of $L_g$ on Einstein $4$-manifolds, and gave explicit representations of ADM mass for warped product manifolds $\Sp^2 \times \Sp^2, \RP^2 \times \RP^2$, etc.   
Recently, Gursky-Malchiodi \cite[Theorem 1.1]{GM} proved a nice mass formula on Einstein $4$-manifolds related to the Green function of $L_g$, subsequently under appropriate topology conditions the classification of Einstein $4$-manifolds with positive sectional curvature was independently studied by L. Cheng \cite{ChengL} and Gursky-Malchiodi \cite{GM2026}. See also Manea-Reznikov \cite{Manea-Reznikov} for the sharp Hessian estimates for Green functions on  Einstein $n$-manifolds with positive sectional curvature.

\begin{convention}
Throughout the paper, let  $r(\cdot)=\mathrm{d}_g(\cdot,\mathsf{p})$, $V=\mathrm{Vol}(M,g)$ and $\lambda_n=n(n-2)/4$. Write $\widehat M=M \setminus\{\mathsf{p}\}$ and  $\hat g=G^{4/(n-2)} g=\varphi^2 g$.
\end{convention}

\begin{theoremletter}[Gursky-Malchiodi]
Let $(M^4,g)$ be an Einstein manifold with $\Ric=3 g$. Let $G$ denote the Green function of $L_g=-\Delta_g +2$ with pole at $\mathsf{p}$. Then 
\begin{equation*}
F:=G^{-1} (1 + \vert \hat{\nabla} G \vert^2) -4 \leq 0 \qquad \mathrm{in~~} \widehat M
\end{equation*}
and the ADM mass $m(\hat g)$ is given by
\begin{align*}
8 |\Sp^3| m(\hat g)=6(|\Sp^4|-V)+\int_{\widehat M} |\hat \nabla F|^2 \ud \mu_{\hat g}+2\int_{\widehat M} G^{-2}|\hat \nabla G|^2 \vert \mathring{\hat{\nabla}}^2G \vert^2\ud \mu_{\hat g}.
\end{align*}
\end{theoremletter}

Inspired by Gursky-Malchiodi's work \cite{GM},  in this paper we continue to explore such a mass formula on Einstein manifolds of dimension 5 and higher. However, such a mass formula on $n$-dimensional Einstein manifolds with $n\geq 6$ cannot be unconditionally valid. So, it is natural for us to impose the `\emph{Weyl vanishing condition}' :\footnote{The \emph{Weyl vanishing condition} was also used in the convergence of the Yamabe flow by Brendle \cite{Brendle} in dimension $n\geq 6$. 
}
\begin{equation}\label{cond:Weyl-vanishing}
n \geq 6 \quad \mathrm{and} \quad \limsup_{x \to \mathsf{p}} \ud_g(x,\mathsf{p})^{2-\lfloor\frac{n-2}{2}\rfloor}|W_g(x)|=0.
\end{equation}
Our first main theorem extends Gursky-Malchiodi mass formula to higher dimensional Einstein manifolds.
\begin{theorem}\label{Thm:G-M type mass formula}
Let $(M^n,g), n \geq 5$ be an Einstein manifold with $\Ric=(n-1) g$. Let $G$ denote the Green function of $L_g=-\Delta_g +\lambda_n$ with pole at $\mathsf{p}$. Assume  either $n=5$ or $n\geq 6$ and  the Weyl vanishing condition \eqref{cond:Weyl-vanishing}  is true. Then
\begin{equation}\label{est:lbd_F}
F:=\varphi^{-1} (1 + \vert \hat{\nabla} \varphi \vert^2) -4 \leq 0 \qquad \mathrm{in~~} \widehat M
\end{equation}
and the ADM mass $m(\hat g)$ is given by
\begin{equation}\label{G-M mass identity}
8\vert \mathbb{S}^{n-1}\vert m(\hat g)
=\int_{\widehat M} \varphi^{-1}\hat \Delta F \ud \mu_{\hat g}+\int_{\widehat M} |\hat \nabla F|^2 \ud \mu_{\hat g}+2\int_{\widehat M} \varphi^{-2}|\hat \nabla \varphi|^2 \vert \mathring{\hat{\nabla}}^2\varphi \vert^2\ud \mu_{\hat g}
\end{equation}
with $\hat \Delta F\geq 0$ in $\widehat M$.
Moreover, there exists a dimensional constant $\alpha_n>0$ such that
\begin{equation}\label{lbd-mass:volume gap}
m(\hat g)\geq \alpha_n (|\Sn|-V).
\end{equation}
\end{theorem}

The same argument in \cite[Proposition 3.3]{GM} applies to show that the equality holds in  \eqref{est:lbd_F} if and only if the metric $\hat g$ is flat.  In a general setting of a closed manifold $(M^n,g), n\geq 3$ with $\Ric \geq (n-1) g$, Manea \cite{Manea} also proved the same gradient estimate \eqref{est:lbd_F} by adapting an approach by Colding \cite{Colding} for nonparabolic manifolds with nonnegative Ricci curvature. 

\medskip

To state our second main theorem of the sharp mass-volume inequality on Einstein manifolds,  we introduce some notations.

Let $(M^n,g), n \geq 4$ be an Einstein manifold with $\Ric=(n-1) g$. The volume of a geodesic ball on $\Sn$ of radius $t$ is
\begin{equation}\label{eq:modelvolume}
 V_n(t)=|\Sp^{n-1}|\int_0^t\sin^{n-1}s\ud s,\qquad  t \in (0,\pi].
\end{equation}
Then there exists a unique real number $\mathfrak R=\fR(M,g)\in(0,\pi]$ such that
\begin{equation}\label{eq:volumeradius}
 V_n(\fR)=V.
\end{equation}
For convenience, we call $\fR$ the volume radius of $(M,g)$.
By Bishop-Gromov volume comparison theorem and the diameter estimate due to Myers \cite{Myers} and S. Y. Cheng \cite{Cheng} , we know that 
$$\fR\leq \mathrm{diam}(M,g) \leq \pi.$$

Define 
$$\sigma(t) = \Big(2\sin (\frac{t}{2})\Big)^{2-n}, \qquad t \in [0,\fR]$$
and its companion function by
$$\zeta (t):=\sigma(\pi-t)=\Big(2\cos (\frac{t}{2})\Big)^{2-n}.$$
 Geometrically, if we regard $t$ as the distance function to $\mathsf p$ on $\Sn$, then $\sigma(t)$ is exactly the Green function of $L_{\Sn}$, precisely,
\begin{equation}\label{eq:Green-fcn-sphere}
-\sigma''-(n-1)\cot t \sigma'+\lambda_n \sigma=0 \qquad \mathrm{for~~} t\in (0,\pi).
\end{equation}
Clearly, $\zeta(t)$ is the Green function with pole at $-\mathsf p$, and satisfies \eqref{eq:Green-fcn-sphere} as well. 

The Bishop-Gromov volume comparison theorem implies that a local volume gap
 \begin{equation}\label{eq:volumedeficit}
 D(t):=V_n(t)-V_{\mathsf p}(t) \geq 0,\qquad  t \in [0,\fR],
\end{equation}
where $V_{\mathsf p}(t)=\Vol_g(B_t(\mathsf p))$.

Let
\begin{equation} \label{eq:fR}
\begin{aligned}
 f_{\fR}(t)=\sigma(t)+k_{\fR} \zeta(t) \qquad \mathrm{~~with~~} k_{\fR}=\cot^n(\fR/2), \quad 0<t\le \fR.
\end{aligned}
\end{equation}
Also $f_{\fR}$ solves \eqref{eq:Green-fcn-sphere} in $(0,\pi)$.
 Extend the function $f_{\fR}$  to a Lipschitz function in $\widehat M$ by a constant
outside $B_{\fR}({\mathsf p})$:
\begin{equation}\label{eq:FRw}
 F_{\fR}(x)=
 \begin{cases}
 f_{\fR}(r(x)),&r(x)\le \fR,\\
 f_{\fR}(\fR),&r(x)>\fR,
 \end{cases}
\end{equation}
and set $w=G-F_{\fR}$. 

Let
$$E[u]=\int_M \bigl( |\nabla u|^2+\lambda_n u^2\bigr) \ud \mu_g, \quad u \in H^1(M,g)$$
denote the Yamabe energy on $(M,g)$, and define a radial Yamabe energy density by
\begin{equation}\label{eq:energy-density}
 e_{\fR}(t)=f_{\fR}'(t)^2+\lambda_n f_{\fR}(t)^2.
 \end{equation}
 
 \begin{theorem}\label{thm:diameter-mass inequality}
Let $(M^n,g), n \geq 4$ be an Einstein manifold with $\Ric=(n-1) g$. Let $G$ denote the Green function of $L_g=-\Delta_g +\lambda_n$ with pole at $\mathsf{p}$. Assume that  the Weyl vanishing condition \eqref{cond:Weyl-vanishing}  is true when $n\geq 6$.
Then $w\in H^1(M,g)$, and
\begin{equation}\label{eq:capidentity} 
 \frac{n-2}{4(n-1)}|\Sp^{n-1} |\left(m(\hat g)-(n-1)2^{4-n}\cot^n\frac{\fR}{2}\right)\\
 =-\int_0^{\fR} D(t)e_{\fR}'(t)\ud t
       +E[w]
\end{equation}
and $D(t)\ge0, e_{\fR}'(t)<0$ for $0<t<\fR$. In particular, we have
\begin{equation}\label{eq:massdiam}
 m(\hat g)\ge (n-1)2^{4-n}\cot^n(\fR/2)\ge (n-1)2^{4-n}\cot^n(\diam(M,g)/2),
\end{equation}
with equality  if and only if $(M, g)$ is isometric to either
the unit round  sphere  $\Sn$ or the round real projective space $\RP^n$.
\end{theorem}

As two applications, we first prove a mass gap theorem for Einstein manifolds, see also Kim-Park \cite[Corollary 1.3]{KimPark} for a related result.

\begin{corollary}\label{Cor:Gap theorem}
Assume the same hypotheses as in Theorem \ref{thm:diameter-mass inequality}. There exists a constant $\delta_n>0$ such that if $m(\hat g)<\delta_n$, then $m(\hat g)=0$, and $(M, g)$ is isometric to $(\Sn, g_{\Sn})$.
\end{corollary}

Another application of Theorem \ref{thm:diameter-mass inequality} is to characterize volume collapsing family of normalized Einstein manifolds.  In \cite[p.1404]{CollSzek}, Collins-Sz\'ekelyhidi constructed a family of Sasaki-Einstein metrics $g_i$ on $\Sp^5$ with $\Ric_{g_i}=4g_i$ while $V_{g_i}\to 0$. Here we changed their notation a little bit. See also the introduction of \cite{GM}.  Recently, T. Ozuch \cite{Ozuch} also constructed 4-dimensional examples of volume collapsing normalized Einstein manifolds.  For these examples, we can prove that their mass sequence blows up uniformly, independent of the pole of Green function. In fact, we can prove a more general result:
\begin{corollary}\label{Cor:mass_blowup}
If $(M_i, g_i, \mathsf{p}_i)$ is a sequence of pointed Einstein $n$-manifolds such that $\Ric_{g_i}=(n-1)g_i$, satisfying the Weyl vanishing condition at $\mathsf{p}_i$ when $n\geq 6$. If their volume $V_{g_i}\to 0$, then the sequence of masses $m(\widehat{g_i})$ blows up uniformly, in a rate independent of $\mathsf{p}_i$. Moreover, we have
$$\liminf_{i\to\infty} V_{g_i} m(\widehat{g_i})\geq \frac{16(n-1)}{n}|\Sp^{n-1}|.$$
\end{corollary}

The paper is organized as follows. In Section \ref{sec:expansion},  we first prove a  Green function expansion of $L_g$ on  an Einstein manifold of dimension $n \geq 5$, generalizing the expansion by Viaclovsky \cite{Viaclovsky} in dimension 4. This enables us to complete the proof of Theorem \ref{Thm:G-M type mass formula} in Section \ref{sec:mass_formula}.  Finally, we prove the sharp mass-volume inequality in Theorem \ref{thm:diameter-mass inequality} and its Corollaries \ref{Cor:Gap theorem}-\ref{Cor:mass_blowup} in Section \ref{sec:d-m_ineq}.

\medskip
\noindent{\textbf{On the AI assistance.} The authors initiated this project in 2025, inspired by Gursky and Malchiodi \cite{GM}. A large portion of the paper is completed without using AI. AI-assisted tools
helped us to finish the proof of Lemma \ref{lem:key-est} that for $n\geq 5$, the term $-\int_M F G \ud \mu_g$ has a lower bound in terms of the volume gap, by suggesting the auxiliary functions; and also suggested the two-pole Green function used in Theorem \ref{thm:diameter-mass inequality}. The authors verified all arguments and take full responsibility for the completeness and correctness of the proofs. }

\section{Green function expansion on Einstein manifolds}\label{sec:expansion}

As a continuation of \cite[Proposition 2.1]{Viaclovsky}, we establish the Green function expansion of $L_g$ on  an Einstein manifold of dimension $n \geq 5$.

Let $f=O_{m,\alpha}(|x|^\beta)$  mean that $f \in C^{m,\alpha}$ and $|\pa^k f|=O(|x|^{\beta-k})$ for $0 \leq k \leq m$.  Then main result of this section is the following:
\begin{theorem}\label{thm:Green_expan}
Let $(M^n,g), n \geq 5$ be an Einstein manifold with $\Ric=(n-1) g$. Let $G$ denote the Green function of $L_g=-\Delta_g +\lambda_n$ with pole at $\mathsf{p}$.  Assume that  the Weyl vanishing condition \eqref{cond:Weyl-vanishing}  is true. Then 
\begin{enumerate}[(1)]
\item For $n=5$, under $g$-normal coordinates around $\mathrm{p}$, we have
\begin{align}\label{expan_G-dim 5} 
G(x)=&|x|^{-3} + \frac{1}{8} |x|^{-1} +A+  \Big(-\frac{1}{720} \vert \Rm \vert^2 + \frac{179}{5760}\Big) |x| \no\\
 &-\frac{1}{180}R_{iabc} R_{jcab} x_i x_j |x|^{-1}  + \frac{1}{360} R_{aijb} R_{akl b} x_i x_j x_k x_l |x|^{-3}\no\\
 &+ \sum_i a_i x_i+ O_{1,\alpha}(|x|^{1+\alpha})
\end{align}
for all $\alpha \in (0,1)$, where $A,a_i \in \RR$. Moreover, $\hat{g}$ 
is asymptotically flat of order $2$ and 
$m(\hat{g}) =16A$.
\item For $n\geq 6$, under $g$-normal coordinates around $\mathsf{p}$,
\begin{equation}\label{eq:regular-remainder} 
G(x) = \sigma(|x|) + A + v(x),\quad \mathrm{with~~}\quad \vert v \vert + |x| \vert \nabla v \vert + |x|^2 \vert \nabla^2 v\vert \to 0,
\end{equation}
where $\sigma(r) = (2\sin (\frac{r}{2}))^{2-n}$and  $A \in \RR$. Moreover, $\hat{g}$ 
is asymptotically flat of order $\tau = \min\{q+2,n-2\} $ for some integer $q>\lfloor\frac{n-6}{2}\rfloor$, and 
$m(\hat{g}) =4(n-1)A$.
\end{enumerate}
\end{theorem}

We start with a general fact: Suppose $(M^n,g),n\geq 4$ is a closed Einstein manifold of positive scalar curvature, normalized as $\Ric = (n-1)g$. 
Under $g$-normal coordinates around $\mathsf{p} \in M$, by \cite[(5.4) and Lemma 5.5]{Lee-Parker} we have 
\begin{align*}
g_{pq} ={}& \delta_{pq} + \frac{1}{3} R_{pijq} x_i x_j + \frac{1}{6} R_{pijq,k} x_i x_j x_k \\
& + \left( \frac{1}{20} R_{pijq,kl} + \frac{2}{45} R_{pijm} R_{qkl m} \right) x_i x_j x_k x_l + O(r^5)
\end{align*}
and
\begin{align*}
\det g =& 1 - \frac{1}{3} R_{ij} x_i x_j - \frac{1}{6} R_{ij,k} x_i x_j x_k \\
& - \left( \frac{1}{20} R_{ij,kl} + \frac{1}{90} R_{pijm} R_{qklm} - \frac{1}{18} R_{ij} R_{kl} \right) x_i x_j x_k x_l + O(r^5).
\end{align*}
Here, all curvatures are evaluated at $\mathsf{p}$. Then it follows that
\begin{align*}
g^{pq} ={}& \delta_{pq} - \frac{1}{3} R_{pijq} x_i x_j - \frac{1}{6} R_{pijq,k} x_i x_j x_k \\
& - \left( \frac{1}{20} R_{pijq,kl} - \frac{1}{15} R_{pijm} R_{qkl m} \right) x_i x_j x_k x_l + O(r^5).
\end{align*}

Now we separate the proof of Theorem \ref{thm:Green_expan} into two cases. 

\subsection{Expansion in dimension 5}
We now fix $n=5$.

\begin{proof}[Proof of Theorem \ref{thm:Green_expan}(1).]

Observe that
\begin{align*}
\det g=& 1 - \frac{4}{3}r^2 + \frac{8}{9}r^4 - \frac{1}{90}R_{pijm} R_{qklm}x_i x_j x_k x_l + O(r^5),\\
\log \sqrt{\det g} &= \frac{1}{2} \left[ -\frac{4}{3} r^2 + \frac{8}{9} r^4 - \frac{1}{90} R_{aijb} R_{aklb} x_i x_j x_k x_l - \frac{1}{2} \left(-\frac{4}{3} r^2\right)^2 + O(r^5) \right] \\
&= -\frac{2}{3} r^2 - \frac{1}{180} R_{aijb} R_{aklb} x_i x_j x_k x_l + O(r^5),
\end{align*}
thereby,
\begin{align*}
\partial_p \log \sqrt{\det g} &= -\frac{4}{3} x_p - \frac{1}{45} R_{apjb} R_{aklb} x_j x_k x_l + O(r^4).
\end{align*}

We introduce an approximate Green function by
\begin{equation}\label{approx_Green_fcn}
u(x) = r^{-3} + \frac{1}{8} r^{-1} + B r  
+ C R_{iabc} R_{jcab} x_i x_j r^{-1} + D R_{aijb} R_{akl b} x_i x_j x_k x_l r^{-3},
\end{equation}
where $B,C,D \in \RR$ are to be determined.
A direct computation yields
\begin{align*}
\partial_q u ={}&-3 r^{-5} x_q - \frac{1}{8} r^{-3} x_q + B r^{-1} x_q \\
& + C \left( R_{qabc} R_{icab} x_i r^{-1} + R_{iabc} R_{qcab} x_i r^{-1} - R_{iabc} R_{jcab} x_i x_j x_q r^{-3} \right) \\
& + D \left( 4 R_{aijb} R_{akqb} x_i x_j x_k r^{-3} - 3 R_{aijb} R_{akl b} x_i x_j x_k x_l x_q r^{-5} \right)\\
={}& -3 r^{-5} x_q - \frac{1}{8} r^{-3} x_q + B r^{-1} x_q \\
& - C \left( R_{iabc} R_{qabc} x_i r^{-1} +  R_{iabc} R_{jcab} x_i x_j x_q r^{-3} \right) \\
& + D \left( 4 R_{aijb} R_{akqb} x_i x_j x_k r^{-3} - 3 R_{aijb} R_{akl b} x_i x_j x_k x_l x_q r^{-5} \right),
\end{align*}
where the second identity follows from the first Bianchi identity
\begin{equation}\label{first_Bianchi}
R_{qabc} R_{icab} x_i + R_{iabc} R_{qcab} x_i=(R_{icab}+R_{ibca})R_{qabc} x_i  =-R_{iabc} R_{qabc} x_i.
\end{equation}
Hence, we obtain
\begin{align}\label{term-1}
g^{pq} \partial_q u =& -3 r^{-5} x_p - \frac{1}{8} r^{-3} x_p + B r^{-1} x_p \no\\
& - C \left( R_{iabc} R_{pabc} x_i r^{-1} + R_{iabc} R_{pcab} x_i x_j x_p r^{-3} \right) \no\\
& + D \left( 4 R_{aijb} R_{akpb} x_i x_j x_k r^{-3} - 3 R_{aijb} R_{akl b} x_i x_j x_k x_l x_p r^{-5} \right) + O(r),
\end{align}
thereby,
$$
\partial_p \log \sqrt{\det g} \, g^{pq} \partial_q u
= 4 r^{-3} + \frac{1}{6} r^{-1} + \frac{1}{15} R_{apjb} R_{akl b} x_j x_k x_l x_p r^{-5} + O(r).
$$

By \eqref{first_Bianchi} we have
\begin{align*}
\pa_p(R_{iabc} R_{jcab}x_i x_j)=R_{pabc}R_{icab}x_i+R_{iabc}R_{pcab}x_i=-R_{iabc} R_{pabc} x_i.
\end{align*}
Substituting the above equation into \eqref{term-1} yields
\begin{align*}
\partial_p (g^{pq} \partial_q u)
={}& -\frac{1}{4} r^{-3} + 4B r^{-1} \\
& - C \Big( R_{pabc} R_{pabc} r^{-1} - R_{pabc} R_{iabc} x_i x_p r^{-3} \\
&\quad\quad - R_{pabc} R_{iabc} x_i x_p r^{-3} + 2 R_{iabc} R_{jcab} x_i x_j r^{-3} \Big) \\
& + D \Big( 4 R_{apib} R_{akpb} x_i x_k r^{-3} + 4 R_{aipb} R_{akpb} x_i x_k r^{-3}  \\
&\quad\quad -4 R_{aijb} R_{ab} x_i x_j r^{-3} -12 R_{aijb} R_{akpb} x_i x_j x_k x_p r^{-5} \\
&\quad\quad -12R_{aijb} R_{akpb} x_i x_j x_k x_p r^{-5}\Big) + O(1) \\
={}& -\frac{1}{4} r^{-3} + (4B - C\vert \Rm \vert^2 + 64D) r^{-1}\\
&+ (2C + 4D) R_{iabc}R_{jabc} x_i x_j r^{-3} \\
&- (2C + 4D) R_{iabc} R_{jcab} x_i x_j r^{-3} \\
&- 24D R_{aijb}R_{aklb} x_i x_j x_k x_l r^{-5} + O(1). 
\end{align*}

Consequently, we are now ready to compute
\begin{align*}
L_g(u) ={}& -\Delta_g u + \frac{15}{4} u \\
={}& -\frac{1}{\sqrt{\det g}} \partial_p ( \sqrt{\det g} g^{pq} \partial_q u ) + \frac{15}{4} u \\
={}& -\partial_p \log \sqrt{\det g} g^{pq} \partial_q u - \partial_p (g^{pq} \partial_q u) + \frac{15}{4} u\\
={}& (24D - \frac{1}{15}) R_{aijb} R_{aklb} x_i x_j x_k x_l r^{-5} \\
& + (2C + 4D) R_{iabc} R_{jcab} x_i x_j r^{-3} \\
& - (2C + 4D) R_{iabc} R_{jabc} x_i x_j r^{-3} \\
& - (4B - C |\Rm|^2 + 64D - \frac{29}{96}) r^{-1} + O(1).
\end{align*}
By letting 
$$
B = -\frac{1}{720} \vert \Rm \vert^2 + \frac{179}{5760}, \quad C=-\frac{1}{180}, \quad D= \frac{1}{360},
$$
 it follows that 
 \begin{equation}\label{conf-Lap:u}
 L_g u= O(1) \qquad \mathrm{~~in~~} \widehat M. 
 \end{equation}
 This directly leads to
 \begin{equation}\label{conf-Lap:G-u}
 L_g(G-u)=O(1) \qquad \mathrm{~~in~~} M. 
  \end{equation}
By the Sobolev embedding theorem we conclude that for any $\alpha \in (0,1)$,
$$
G(x) = u(x) + A + \sum_i a_i x_i + O_{1,\alpha}(r^{1+\alpha}), \qquad A,a_i\in \RR.
$$
This in turn gives the expansion formula \eqref{expan_G-dim 5}.\\

It remains to prove the asymptotic order and to compute the mass.  For this, consider the inverted normal coordinates $y = \vert x \vert^{-2}x$ and let $t = \vert y \vert$. Then it suffices to expand
\begin{equation}\label{Green-fcn-y}
G(y) = t^3 \Big(1+\frac{1}{8}t^{-2}  + A t^{-3} + O_1(t^{-4})\Big),
\end{equation}
thereby,
$$
G(y)^{\frac{4}{3}} = t^{4} \Big(1 + \frac{1}{6}t^{-2}  + \frac{4}{3}A t^{-3}+ O_1(t^{-4})\Big).
$$

As in \cite{Chen-Shi} and \cite{CGS}, it is convenient to use the mass formula associated to asymptotic polar coordinates (see, e.g. Lee-Parker \cite[Equation (9.6) on p.79]{Lee-Parker}) instead of \eqref{Def:ADM_mass}. This greatly simplifies the calculations.

Under the asymptotic coordinates $y$ we have
\begin{align*} 
g ={}&(\delta_{ij} - \frac{1}{3} R_{ikjl} \frac{y_k y_l}{t^4} - \frac{1}{6} R_{ikjl,m} \frac{y_k y_m y_l}{t^6} + O(t^{-4})) \\
& \cdot\frac{1}{t^4} (\md y_i - \frac{2y_i}{t}\md t) \otimes (\md y_j - \frac{2y_j}{t}\md t)  \\
={}& \frac{1}{t^4} \md y_i \otimes \md y_i - \frac{1}{3} R_{ikjl} \frac{y_k y_l}{t^8} \md y_i \otimes \md y_j - \frac{1}{6} R_{ikjl,m} \frac{y_k y_m y_l}{t^{10}} \md y_i \otimes \md y_j + O(t^{-8})
\end{align*}
and
\begin{align}\label{hatg}
\hat{g}_{ij} ={}& \bigg(1 + \frac{1}{6}t^{-2}+ \frac{4}{3}A t^{-3} + O_1(t^{-4})\bigg)\no\\
&\cdot \bigg( \delta_{ij} - \frac{1}{3} R_{ikjl} \frac{y_k y_l}{t^4}  - \frac{1}{6} R_{ikjl,m} \frac{y_k y_m y_l}{t^6} + O(t^{-4})\bigg)\no \\
={}&(1 + \frac{1}{6}t^{-2}+ \frac{4}{3}A t^{-3}) \delta_{ij} -\frac{1}{3} R_{ikjl} y_k y_l t^{-4}  - \frac{1}{6} R_{ikjl,m} y_k y_m y_l t^{-6}+ O_1(t^{-4}).
\end{align}
This implies that $(\widehat M, \hat g)$ is asymptotical flat manifold of order $2$.

A direct computation yields
\begin{align*}
\sum_i \hat{g}_{ii} &= n(1 + \frac{1}{6}t^{-2}+ \frac{4}{3}A t^{-3}) - \frac{1}{3} R_{kl} y_k y_l t^{-4} + R_{kl,m} y_k y_l y_m t^{-6} + O_1(t^{-4}) \\
&= n(1 + \frac{1}{6}t^{-2}+ \frac{4}{3}A t^{-3}) - \frac{1}{3n} R t^{-2}  + O_1(t^{-4}) \\
&= 5(1 + \frac{1}{6}t^{-2}+ \frac{4}{3}A t^{-3}) - \frac{1}{15} R t^{-2}  + O_1(t^{-4}), \\
\hat{g}_{tt} &= \sum_{i,j}\hat{g}_{ij} \frac{y_i y_j}{t^2} = 1 + \frac{1}{6}t^{-2}+ \frac{4}{3}A t^{-3} + O_1(t^{-4}).
\end{align*}
Observe that
\begin{align*}
t^{-1}(5\hat{g}_{tt} - \sum_i \hat{g}_{ii}) &= \frac{1}{15} R t^{-3}  + O(t^{-5}) \\
\pa_t (\hat{g}_{tt} - \sum_i \hat{g}_{ii}) &= \pa_t \big(-4(1 + \frac{1}{6}t^{-2}+ \frac{4}{3}A t^{-3}) + \frac{1}{15} R t^{-2} + O_1 (t^{-4} ) \big) \\
&= \frac{4}{3}t^{-3} - \frac{2}{15} R t^{-3} + 16A t^{-4} +  O(t^{-5}).
\end{align*}
Hence, we obtain
\begin{align*}
m(\hat{g}) &= \frac{1}{\vert \mathbb{S}^{4} \vert} \lim_{t \to \infty} \int_{\mathbb{S}^{n-1}_t} t^{-1}(5\hat{g}_{tt} - \sum_i \hat{g}_{ii}) + \pa_t (\hat{g}_{tt} - \sum_i \hat{g}_{ii}) \ud \sigma \\
&= \frac{1}{\vert \mathbb{S}^{4} \vert} \lim_{t \to \infty} \int_{\mathbb{S}^{n-1}_t} (\frac{4}{3} - \frac{R}{15}) t^{-3} + 16A t^{-4} + O(t^{-5})  \ud \sigma \\
&= 16A.
\end{align*}
\end{proof}

\subsection{Expansion in dimension 6 and higher}

Now we assume $n\geq 6$,  and assume the Weyl vanishing condition \eqref{cond:Weyl-vanishing}.

\begin{proof}[Proof of Theorem \ref{thm:Green_expan}(2).]

Consider a geodesic ball $B_\delta (\mathsf{p})$ and let $N = \nabla r$. Then $\vert N \vert_g = 1$ and $\nabla_N N = 0$. Define the shape operator by
$$
SX:=\nabla_X N  \quad \mathrm{~~for~~} \langle X,N \rangle=0.
$$
which is symmetric because $\langle SX,Y\rangle= \nabla^2r(X,Y)$.   We use the convention for curvature tensor that
$$
\R(X,Y)U = \nabla_X \nabla_Y U -\nabla_Y \nabla_X U -\nabla_{[X,Y]} U.
$$
We define the radial curvature endomorphism by
$R_r X=\R(X,N)N$ for $\langle X,N\rangle=0$.
Direct computation gives us the  Riccati equation
\begin{equation}\label{eq:shape-riccati}
 S' + S^2 + R_r=0.
\end{equation}
On the other hand, the Schouten tensor $P=\frac{1}{2}g$ by the Einstein condition $\Ric=(n-1)g$. Then
the Riemann curvature decomposition gives
\begin{align*}
\R(X,Y)U={}&W(X,Y)U +P(Y,U)X-P(X,U)Y\\
&+\langle Y,U\rangle   P^\sharp X-\langle X,U\rangle P^\sharp Y\\
={}&W(X,Y)U+\langle Y,U \rangle X-\langle X,U \rangle Y.
\end{align*}
Here $ P^\sharp$ is the raising one index of the Schouten tensor $P$ with $g$.
Set $Y=U=N$ and $\langle X,N\rangle = 0$, it gives
\begin{equation}\label{eq:radial-curvature}
 R_r=I+W_r,\quad W_rX:=W(X,N)N.
\end{equation}
For an orthonormal basis $E_1,\cdots,E_{n-1}$ of $N^\perp$,
\begin{equation}\label{eq:radial-weyl-trace}
 \tr W_r
 =\sum_{a=1}^{n-1}\langle \R(N,E_a)N,E_a\rangle-(n-1)
 =\Ric(N,N)-(n-1)=0.
\end{equation}
Let 
\begin{equation}\label{eq:def of Z}
Z=S-\cot r\,I
\end{equation}
be a $(1,1)$-tensor on $\pa B_\delta(\mathsf p)$, where  $\cot r\,I$ is the shape operator on a geodesic sphere of radius $r$ in $\Sn$. 
Plugging \eqref{eq:radial-curvature} into
\eqref{eq:shape-riccati} yields
\begin{align*}
Z' + Z^2 &= (S' + S^2) - (\cot r)' I + \cot^2 r I - 2 \cot r S \\
&= - (I + W_r) + (1 + \cot^2 r ) I  +\cot^2 r I - 2 \cot r S\\
&= - W_r - 2\cot r ( S - \cot r )\\ 
&= -W_r - 2\cot r Z
\end{align*}
and 
\begin{align*}
(\sin^2 r Z)' &= 2\sin r \cos r~ Z + \sin^2 r Z' \\
&= 2\sin r \cos r~Z - \sin^2 r~(2\cot r Z + Z^2 + W_r)\\
&= -\sin^2 r~(Z^2 + W_r).
\end{align*}
Thus we obtain
\begin{equation}\label{eq:equation of Z}
 Z'+2\cot r\,Z+Z^2+W_r=0,
 \qquad (\sin^2r\,Z)'=-\sin^2r\,(Z^2+W_r).
\end{equation}
In geodesic normal coordinates, 
$$
\nabla^2_{\alpha \beta } r = \pa^2_{\alpha\beta} r - \Gamma_{\alpha\beta}^\gamma \pa_\gamma r = \frac{\delta_{\alpha\beta}}{r} - \frac{x_\alpha x_\beta}{r^3} + O(r),
$$
which yields the expansion for the shape operator $S=r^{-1}I+O(r)$. This in turn implies $Z=O(r)$.
Hence, integrating \eqref{eq:equation of Z} over $(0,r)$ yields
\begin{equation}\label{eq:Z}
 Z(r)=-\frac{1}{\sin^2r}\int_0^r\sin^2t\,
       \bigl(Z(t)^2+W_t\bigr)~\ud t.
\end{equation}
From the Weyl vanishing condition \eqref{cond:Weyl-vanishing} that  $W_t=O(t^q)$ for some integer $q>\lfloor \frac{n-6}{2}\rfloor$. If $Z=O(r^m)$, then the above identity improves the estimate to
$Z=O(r^{\min\{q+1, 2m+1\}})$. Iterating from $m=1$  gives
\begin{equation}\label{Z:vanishing-order}
Z=O(r^{q+1}).
\end{equation} Taking the trace in \eqref{eq:Z} and using \eqref{eq:radial-weyl-trace}, we obtain a stronger estimate
\begin{equation}\label{eq:e}
 \tr Z
 =-\frac1{\sin^2r}\int_0^r\sin^2t\,|Z(t)|^2 \ud t
 =O(r^{2q+3}).
\end{equation}

Under the geodesic polar coordinates $(r,\theta) \in \RR_+\times \Sp^{n-1}$,
\begin{equation*}
g = \md r^2+g_r \qquad \mathrm{with~~}g_r(\theta)=\sin^2r\,a(r,\theta), ~~a(0,\theta)=g_{\mathbb{S}^{n-1}}.
\end{equation*}
Also $g_{\Sn}=\md r^2 +\sin^2 r g_{\Sp^{n-1}}$. By definition \eqref{eq:def of Z} of $Z$ we have
\begin{equation}\label{eq:r-derivative of metric}
\begin{aligned}
    \pa_r(g_r-\sin^2 r g_{\Sp^{n-1}})(X, Y)&=2g_r(SX, Y)-2\sin r\cos r g_{\Sp^{n-1}}(X, Y)\\
    &=2g_r (ZX, Y)+2g_r(\cot r X, Y)-2\sin r\cos r g_{\Sp^{n-1}}(X, Y)\\
    &=2g_r (ZX, Y)+2\cot r\left(g_r-\sin^2rg_{\Sp^{n-1}}\right)(X, Y).
\end{aligned}
\end{equation}
Set $Z^b(X, Y)=g_r(ZX, Y)$. Then \eqref{eq:r-derivative of metric} becomes
\begin{equation}\label{ODE:diff-level-set-metric}
    \pa_r(g_r-\sin^2 r g_{\Sp^{n-1}})-2\cot r\left(g_r-\sin^2rg_{\Sp^{n-1}}\right)=2Z^b.
\end{equation}
So,
\begin{equation}
\begin{aligned}
\pa_ra=&\pa_r(\sin^{-2}r\left(g_r-\sin^2 r g_{\Sp^{n-1}}\right))\\
=&\sin^{-2}r\left(\pa_r(g_r-\sin^2 r g_{\Sp^{n-1}})-2\cot r(g_r-\sin^2 r g_{\Sp^{n-1}})\right)\\
&=2\sin^{-2}r Z^{b}=2a(Z\cdot, \cdot).
\end{aligned}
\end{equation}
Consequently, we obtain
\begin{equation}\label{eq:log-det}
\partial_r\log\det(g_{\mathbb{S}^{n-1}}^{-1}a)=\pa_r\log \det a=a^{-1}\pa_ra=2\tr Z.
\end{equation}
Rearranging ODE \eqref{ODE:diff-level-set-metric} as
$$\pa_r \bigg(\sin^{-2}r (g_r-\sin^2 r g_{\Sp^{n-1}})\bigg)=\frac{2Z}{\sin^2 r}.$$
Integrating the above equation over $(0,r)$ together with  \eqref{Z:vanishing-order}  yields
\begin{equation}\label{eq:angularbound-1}
 g_r-\sin^2 r g_{\mathbb{S}^{n-1}}=O(r^{q+2}).
\end{equation}
Similarly, by \eqref{eq:log-det} and \eqref{eq:e} we obtain
\begin{equation}\label{eq:angularbound-2}
 \log\det(g_{\mathbb{S}^{n-1}}^{-1}a)=O(r^{2q+4}).
\end{equation}
The expansion of $\log\det(I+E)=\tr E+O(|E|^2)$  gives
\begin{equation}\label{eq:traceangular}
 \tr_{g_{\mathbb{S}^{n-1}}}(a-g_{\mathbb{S}^{n-1}})=O(r^{2q+4}),
\end{equation}
and its derivative estimates hold uniformly in $\theta$.  

By \eqref{eq:Green-fcn-sphere} and \eqref{eq:e} we obtain that away from $\mathsf p$,
\begin{equation*}
 L_g \sigma=-\sigma'(r) \tr Z=O(r^\beta),\qquad \beta:=2q+4-n>-2.
\end{equation*}
In the distribution sense, $L_g \sigma$ contains the term $(n-2)|\Sp^{n-1}|\delta_{\mathsf p}$ as
$\sigma \sim r^{2-n}$ near $\mathsf p$. Take a smooth cutoff function $\chi$ supported in $B_\delta (\mathsf p)$ and equal to $1$ near $\mathsf p$, the Dirac terms in
$L_g(G-\chi \sigma)$ cancel, and its remaining part belongs to $L^s(M)$
for some $s>n/2$. Elliptic regularity theory gives
$G-\chi \sigma \in W^{2,s}(M)\subset C^0(M)$. So, we have the expression of $G=\sigma+A+v$ near $\mathsf p$.

Set $v=G-\sigma-A$. Then $v\to0$ as $r \to 0$ and
$L_g v=\sigma'\tr Z- n(n-2) A/4 $  near $\mathsf p$. Interior Schauder estimates on rescaled annuli (see for example \cite{Gilbarg-Trudinger}) yield
$$
 \sup_{B_{2r}\setminus B_{r/2}}\Big(r|\nabla v|+r^2|\nabla^2v|\Big)
 \le C\Big(\sup_{B_{4r}\setminus B_{r/4}}|v|
             +r^{\beta+2}+r^2|A|\Big)\rightarrow 0 \quad \mathrm{as~~} r \to 0.
$$
This proves (\ref{eq:regular-remainder}).

To discuss the asymptotic order and compute the mass, we introduce the asymptotic coordinates by
\begin{equation*}
 \rho=\tfrac12\cot(r/2),\qquad y=\rho\theta.
\end{equation*}
Using the identities $\rho'=-(2\sin(r/2))^{-2} = -(\rho^2 +1/4)$ and
$\sigma(r)=(\rho^2+1/4)^{(n-2)/2}$, we have
\begin{equation*}
g_0:= \sigma^{4/(n-2)} g = \md\rho^2+\rho^2a(r,\theta) =  \vert \md y \vert^2 - \rho^2 g_{\mathbb{S}^{n-1}} + \rho^2 a(r, \theta)  = \delta+h.
\end{equation*}
By \eqref{eq:angularbound-1}-\eqref{eq:traceangular} we obtain
\begin{equation}\label{eq:hestimates}
 h_{ij}\nu^j=0,\qquad
 h=O_2(\rho^{-q-2}),\qquad
 \tr_\delta h=O_1(\rho^{-2q-4}),\qquad \nu=y/\rho.
\end{equation}
On the other hand, the Green function expansion \eqref{eq:regular-remainder} gives
\begin{equation}\label{eq:factor}
 \left(\frac{G}{\sigma}\right)^{4/(n-2)}
 =1+\frac{4A}{n-2}\rho^{2-n}+o_2(\rho^{2-n}),
\end{equation}
where $o_2(\rho^{2-n})$ means that the derivatives of order $j$ are $o(\rho^{2-n-j})$.
Consequently,
\begin{equation}\label{eq:gexpansion}
 \hat{g}_{ij}= \left(\frac{G}{\sigma}\right)^{4/(n-2)} g_0 = \delta_{ij}+h_{ij} +\frac{4A}{n-2}\rho^{2-n}\delta_{ij} + o_2(\rho^{2-n}).
\end{equation}
Hence, under the asymptotic $y$-coordinates, $\hat{g}$ has asymptotically flat order
$\tau=\min\{q+2,n-2\}>(n-2)/2$.

Observe that
$$
\sum_i \hat g_{ii} = n + \tr_\delta h + \frac{4n A}{n-2} \rho^{2-n} + o_2 (\rho^{2-n} ),
$$
and from \eqref{eq:hestimates}-\eqref{eq:gexpansion} that
$$
\hat g_{\rho\rho} =  1 + \frac{4 A}{n-2} \rho^{2-n} + o_2 (\rho^{2-n} ).
$$
Combining these terms above yields
\begin{align*}
&\pa_\rho (\hat g_{\rho\rho} - \sum_i \hat g_{ii}) + \frac{1}{\rho}(n\hat g_{\rho\rho} - \sum_i \hat g_{ii}) \\
={}& 4(n-1) A \rho^{1-n} + \pa_\rho \tr_\delta h + \frac{\tr_\delta h}{\rho} + o(\rho^{1-n}) \\
={}& 4(n-1) A \rho^{1-n} + O(\rho^{-2q-5}) + o(\rho^{1-n}).
\end{align*}
Consequently, we obtain
\begin{align*}
m(\hat{g}) &= \frac{1}{\vert \mathbb{S}^{n-1} \vert} \lim_{\rho \to \infty} \int_{\vert y \vert = \rho} \pa_\rho (\hat g_{\rho\rho} - \sum_i \hat g_{ii}) + \frac{1}{\rho}(n\hat g_{\rho\rho} - \sum_i \hat g_{ii})~ \ud \sigma\\
&= \frac{1}{\vert \mathbb{S}^{n-1} \vert} \lim_{\rho \to \infty} \int_{\mathbb{S}^{n-1}} 4(n-1) A  + O(\rho^{n-6-2q}) + o(1)~\ud\sigma\\
&= 4(n-1)A.
\end{align*}
This completes the proof.
\end{proof}

\begin{remark}\label{rem:lowdim}
For $n=4,5$, take $W_g=O(1)$ near $\mathsf p$, the same argument starting from \eqref{eq:Z} applies to obtain that
$$G=\sigma (r)+A+v, \qquad \mathrm{with~~} \vert v \vert + |x| \vert \nabla v \vert + |x|^2 \vert \nabla^2 v\vert \to 0 \quad \mathrm{as~~} |x| \to 0.$$
This is also a direct consequence of \cite[Proposition 2.1]{Viaclovsky} and Theorem \ref{thm:Green_expan}(1).
\end{remark}

\section{Gursky-Malchiodi type mass formulas}\label{sec:mass_formula}

We start with computations those are useful both in 5-dimensional and higher dimensional cases.
Let $\widehat M:=M \setminus\{\mathsf{p}\}$ and $\hat{g} = G^{\frac{4}{n-2}}g$ = $\varphi^2 g$ with $\hat R=0$ in $\widehat M$. The conformal covariance of the conformal Laplacian yields
\begin{align*}
-\hat{\Delta} \varphi &= L_{\hat{g}}(\varphi)\\
&= G^{-\frac{n+2}{n-2}} L_g (G \cdot G^{\frac{2}{n-2}}) \\
&= G^{-\frac{n+2}{n-2}} (- \Delta G^{\frac{n}{n-2}} + \frac{n(n-2)}{4} G^{\frac{n}{n-2}}) \\
&= G^{-\frac{n+2}{n-2}}(-\frac{n}{n-2}G^{\frac{2}{n-2}}\Delta G - \frac{n}{n-2} \frac{2}{n-2}G^{\frac{4-n}{n-2}} \vert\nabla G \vert^2 + \frac{n(n-2)}{4}G^{\frac{n}{n-2}}) \\
&=G^{-\frac{n+2}{n-2}}(-\frac{n}{n-2}G^{\frac{2}{n-2}}\frac{n(n-2)}{4}G - \frac{2n}{(n-2)^2} G^{\frac{8-n}{n-2}} \vert\hat{\nabla} G \vert^2 + \frac{n(n-2)}{4}G^{\frac{n}{n-2}}) \\
&= G^{-\frac{n+2}{n-2}}(-\frac{n}{2}G^{\frac{n}{n-2}} - \frac{n}{2}G^{\frac{n}{n-2}}\vert\hat{\nabla} G^{\frac{2}{n-2}} \vert^2) \\
&= -\frac{n}{2} G^{\frac{2}{2-n}}(1+\vert\hat{\nabla} G^{\frac{2}{n-2}} \vert^2).
\end{align*}
This leads to
\begin{equation}\label{Lap_varphi}
\hat{\Delta}\varphi = \frac{n}{2} \varphi^{-1} (1 + \vert \hat{\nabla} \varphi \vert^2 ) \qquad \mathrm{in~~} \widehat M.
\end{equation}
Similarly, since $L_g G=0$ in $\widehat M$, we have
\begin{align*}
-\hat\Delta \varphi^{-1}=L_{\hat g}\varphi^{-1}=&G^{-\frac{n+2}{n-2}}L_g(G\varphi^{-1})\\
=&G^{-\frac{n+2}{n-2}}\Big[-\Delta G^{\frac{n-4}{n-2}}+\frac{n(n-2)}{4}G^{\frac{n-4}{n-2}}\Big]\\
=&G^{-\frac{n+2}{n-2}}\Big[-\frac{n-4}{n-2}G^{-\frac{2}{n-2}}\Delta G +\frac{2(n-4)}{(n-2)^2}G^{-\frac{n}{n-2}}|\nabla G|^2+\frac{n(n-2)}{4}G^{\frac{n-4}{n-2}}\Big]\\
=&G^{-\frac{n+2}{n-2}}\Big[\frac{n}{2}G^{\frac{n-4}{n-2}}+\frac{2(n-4)}{(n-2)^2}G^{\frac{4-n}{n-2}}|\hat \nabla G|^2\Big].
\end{align*}
This implies
\begin{equation}\label{Lap_inverse_varphi}
-\hat\Delta \varphi^{-1}=\varphi^{-3}\Big(\frac{n}{2}+\frac{n-4}{2}|\hat\nabla \varphi|^2\Big) \qquad \mathrm{in~~} \widehat M.
\end{equation}
The conformal change of Schouten tensor $P=\frac{1}{n-2}(\Ric-\frac{R}{2(n-1)}g)$ gives
\begin{align*}
\frac{1}{2} \varphi^{-2} \hat{g}_{ij} &= \frac{1}{2} g_{ij} 
=P_{ij} \\ &= \hat{P}_{ij} + (\log \varphi)_{,ij} + (\log \varphi)_{,i}(\log \varphi)_{,j} - \frac{1}{2} \vert \hat{\nabla} \log \varphi \vert^2 \hat{g}_{ij}\\
&= \hat{P}_{ij} + (\varphi^{-1} \varphi_{,ij} - \varphi^{-2} \varphi_{,i} \varphi_{,j}) + \varphi^{-2} \varphi_{,i} \varphi_{,j} - \frac{1}{2} \varphi^{-2}\vert \hat{\nabla}\varphi \vert^2 ]\hat{g}_{ij} \\
&= \hat{P}_{ij}+ \varphi^{-1} \varphi_{,ij} - \frac{1}{2} \varphi^{-2}\vert \hat{\nabla}\varphi \vert^2 \hat{g}_{ij}.
\end{align*}
Here, the comma means the covariant derivative with respect to the metric $\hat{g}$. Then
\begin{align*}
\hat{P}_{ij} &= - \varphi^{-1} \varphi_{,ij} + \frac{1}{2} \varphi^{-2}(1+\vert \hat{\nabla}\varphi \vert^2 )\hat{g}_{ij} \\
&\overset{(\ref{Lap_varphi})}{=} - \varphi^{-1}(\varphi_{,ij} - \frac{\hat{\Delta}\varphi}{n} \hat{g}_{ij}) = - \varphi^{-1}(\mathring{\hat{\nabla}}^2\varphi)_{ij}.
\end{align*}

Let $F = \varphi^{-1} (1 + \vert \hat{\nabla} \varphi \vert^2) -4$. Then we have
\begin{align*}
F_{,j} &= -\varphi^{-2}\varphi_{,j}(1+\vert \hat{\nabla} \varphi \vert^2) + 2 \varphi^{-1}\varphi_{,ij}\varphi_{,i} \\
&\overset{(\ref{Lap_varphi})}{=} -\frac{2}{n} \varphi^{-1} \varphi_{,j} \hat{\Delta} \varphi  + 2 \varphi^{-1}(\mathring{\hat{\nabla}}^2\varphi)_{ij} \varphi_{,i} + \frac{2}{n} \varphi^{-1} \varphi_{,j} \hat{\Delta} \varphi \\
&= -2 \hat{P}_{ij} \varphi_{,i}
\end{align*}
and hence
\begin{equation}\label{Lap_F}
\hat{\Delta} F = -2 \hat{P}_{ij} \varphi_{,ij} = 2 \varphi^{-1} \vert \mathring{\hat{\nabla}}^2\varphi \vert^2.
\end{equation}
This yields
\begin{align}\label{Lap_Aux-fcn}
\frac{1}{2}\hat \Delta (F^2+8F)=&|\hat \nabla F|^2+(F+4)\hat \Delta F\no\\
=&|\hat \nabla F|^2+\varphi^{-1}\hat \Delta F+2 \varphi^{-2}|\hat \nabla \varphi|^2 \vert \mathring{\hat{\nabla}}^2\varphi \vert^2.
\end{align}

Write $g=G^{-\frac{4}{n-2}}\hat g$. Then we have
\begin{equation}\label{eq:Lap_G-inverse}
-\hat \Delta G^{-1}= \lambda_n G^{-\frac{n+2}{n-2}} \quad \mathrm{~~in~~} \widehat M.
\end{equation}
\begin{remark}
If $(\widehat M,\hat g)=(\Rn,|\md y|^2)$, then by \eqref{eq:Lap_G-inverse} we have
$$\varphi(y)=G(y)^{\frac{2}{n-2}}=|y|^2+\frac{1}{4}$$
and clearly $|\hat \nabla \varphi|^2=4\varphi-1$.
\end{remark}

The proof of Theorem \ref{Thm:G-M type mass formula} consists of the next three subsections.

\subsection{Mass formula in dimension 5}

We first prove the gradient bound \eqref{est:lbd_F} and the mass formula \eqref{G-M mass identity} in Theorem \ref{Thm:G-M type mass formula} when $n=5$.

Write 
$$
G_0(y)=u(y)+A+ t^{-2} \sum_i a_i y_i
$$
with $u$ as in \eqref{approx_Green_fcn}.
Then by \eqref{expan_G-dim 5}, we have
$$
G(y)=G_0(y) + O_{1,\alpha} (t^{-1-\alpha}).
$$
Define
\begin{align}\label{def:varphi_0}
&\varphi_0(y)= G_0^{\frac{2}{3}}(y)= t^2\Big[1 + \frac{1}{12}t^{-2} + \frac{2}{3} At^{-3} + \frac{2}{3} \sum_i a_i y_i t^{-5} + \frac{2}{3}v(y)t^{-3}-\frac{1}{576}t^{-4} \Big], 
\end{align}
where
\begin{align*}
v(y)= \Big(-\frac{1}{720} \vert \Rm \vert^2 + \frac{179}{5760}\Big) t^{-1} -\frac{1}{180}R_{iabc} R_{jcab} y_i y_j t^{-3}  + \frac{1}{360} R_{aijb} R_{akl b} y_i y_j y_k y_l t^{-5}.
\end{align*}
Then it follows that 
$$\varphi(y) = \varphi_0(y) + O_{1, \alpha}(t^{-2-\alpha}).$$

For our purpose, in order to improve the regularity of $G$, we  instead  improve the regularity of $\varphi$ under the inverted normal coordinates.

\begin{lemma}\label{lem:improved-regularity}
Let $\varphi=G^{2/3}$. Then under inverted normal coordinates $y$,
$$\varphi(y)=\varphi_0(y)+O_{2,\alpha}(|y|^{-2-\alpha}),$$
where $\varphi_0$ is given as in \eqref{def:varphi_0}.
\end{lemma}
\begin{proof}
Observe that
\begin{align*}
\hat{\Delta} \varphi_0 &= \frac{1}{\sqrt{\det \hat{g}}} \pa_p (\sqrt{\det \hat{g}} \ \hat{g}^{pq}\pa_q \varphi_0) \\
&=\frac{1}{\varphi^5 \sqrt{\det g}} \pa_p (\varphi^5\sqrt{\det g} \ \varphi^{-2} g^{pq}\pa_q \varphi_0)\\
&=\frac{\varphi^{-1}}{\varphi_0^4 \sqrt{\det g}} \pa_p (\varphi_0^4\sqrt{\det g} \ g^{pq}\pa_q \log \varphi_0) + O_{0,\alpha} (t^{-4-\alpha}) \\
&=\varphi^{-1} \left(\Delta \log \varphi_0 + 4\vert \nabla \log \varphi_0 \vert^2  \right) + O_{0,\alpha} (t^{-4-\alpha}) \\
&=\varphi^{-1} \left(\frac{2}{3} \Delta \log G_0 + \frac{16}{9}\vert \nabla \log G_0 \vert^2  \right) + O_{0,\alpha} (t^{-4-\alpha}) \\
&=\varphi^{-1} \left(\frac{2}{3} G_0^{-1} \Delta G_0 + \frac{10}{9} G_0^{-2} \vert \nabla  G_0 \vert^2  \right) + O_{0,\alpha} (t^{-4-\alpha}).
\end{align*}

Moreover, by \eqref{Lap_varphi} we have
\begin{align*}
\hat{\Delta} \varphi  &= \frac{5}{2} \varphi^{-1} (1 + \vert \hat{\nabla} \varphi \vert^2) \\
&= \frac{5}{2} \varphi^{-1} (1 + \varphi^{-2} \vert \nabla \varphi \vert^2) \\
&= \frac{5}{2} \varphi^{-1} (1 + \varphi_0^{-2} \vert \nabla \varphi_0 \vert^2) + O_{0,\alpha} (t^{-4-\alpha}) \\
&= \frac{5}{2} \varphi^{-1} (1 + \frac{4}{9} G_0^{-2} \vert \nabla G_0 \vert^2) + O_{0,\alpha} (t^{-4-\alpha}). 
\end{align*}

Consequently, we conclude that
$$
\hat{\Delta} (\varphi - \varphi_0) = \frac{2}{3} \varphi^{-1} G_0^{-1} L_g G_0 + O_{0,\alpha} (t^{-4-\alpha}) =O_{0,\alpha} (t^{-4-\alpha}),
$$
where the last identity follows from 
$$L_g (G-G_0) = O(1) \quad \mathrm{by}\quad \eqref{conf-Lap:G-u} \quad  \Rightarrow \quad \varphi^{-1} G_0^{-1} L_g G_0=O(t^{-5}).$$
Hence, the elliptic regularity in asymptotically flat spaces (cf. \cite[Theorem 9.2]{Lee-Parker}) yields $\varphi -\varphi_0 \in C_{-2-\alpha}^{2,\alpha}(\widehat M)$, that is,
$\varphi = \varphi_0 + O_{2, \alpha}(t^{-2-\alpha})$.
\end{proof}

\begin{proof}[Proof of \eqref{est:lbd_F}\eqref{G-M mass identity} in Theorem \ref{Thm:G-M type mass formula} when $n=5$.]
By Green function expansion \eqref{Green-fcn-y} and Theorem \ref{thm:Green_expan}(a),  we obtain
\begin{equation}\label{Asymp_F}
F = -\frac{32}{3} A t^{-3} + o(t^{-3}) = - \frac{2}{3} m(\hat{g}) t^{-3} + o(t^{-3})
\end{equation}
and
$$
F^2 + 8 F =  - \frac{16}{3} m(\hat{g}) t^{-3} + o(t^{-3}).
$$
By \eqref{Lap_F} and \eqref{Asymp_F}, we apply the maximum principle to show that $F \leq 0$ in $\widehat M$.

Define $\widehat M_\rho:=\widehat M \cap \{|y|<\rho\}$ for $\rho \gg 1$. Integrating \eqref{Lap_Aux-fcn} over $\widehat M_\rho$ yields
\begin{align*}
8 |\Sp^4| m(\hat g)+o(1)=&\frac{1}{2}\int_{\pa \widehat M_\rho}\frac{\pa}{\pa \hat \nu}(F^2+8F) \ud \sigma_{\hat g}\\
=&\int_{\widehat M_\rho} |\hat \nabla F|^2 \ud \mu_{\hat g}+2\int_{\widehat M_\rho} \varphi^{-2}|\hat \nabla \varphi|^2 \vert \mathring{\hat{\nabla}}^2\varphi \vert^2\ud \mu_{\hat g}+\int_{\widehat M_\rho} \varphi^{-1}\hat \Delta F \ud \mu_{\hat g}.
\end{align*}
Finally, letting $\rho \to \infty$ we obtain \eqref{G-M mass identity} when $n=5$. 
\end{proof}

\subsection{Mass formula in dimension 6 and higher}

Now we prove the $n\geq 6$ case,  with the Weyl vanishing condition.

\begin{proof}[Proof of \eqref{est:lbd_F}\eqref{G-M mass identity} in Theorem \ref{Thm:G-M type mass formula} when $n\geq 6$.]
By (\ref{eq:gexpansion}) we obtain
\begin{equation}\label{invgexpansion}
\hat{g}^{ij} = \delta_{ij} - h_{ij} -\frac{4A}{n-2}\rho^{2-n}\delta_{ij}+o_2(\rho^{2-n})
\end{equation}
and from (\ref{eq:factor}) that
\begin{align*}
\varphi &= G^{\frac{2}{n-2}} = (\rho^2 + \frac{1}{4})(1 + \frac{2}{n-2} A \rho^{2-n} + o_2(\rho^{2-n})) \\
&= \rho^2 + \frac{1}{4}+\frac{2}{n-2} A \rho^{4-n} + o_2(\rho^{4-n}).
\end{align*}
By \eqref{invgexpansion} and \eqref{eq:hestimates} we have 
\begin{align*}
1+\vert \hat{\nabla} \varphi \vert^2 ={}& 1+\hat{g}^{ij} \varphi_i \varphi_j \\
={}& 1+\Big(\delta_{ij} - h_{ij} -\frac{4A}{n-2}\rho^{2-n}\delta_{ij}+o_2(\rho^{2-n})\Big)\Big(2y_i + \frac{2(4-n)}{n-2}A\rho^{2-n}y_i+ o_1(\rho^{3-n})\Big) \\
&\qquad \cdot \Big(2y_j+ \frac{2(4-n)}{n-2}A\rho^{2-n}y_j+ o_1(\rho^{3-n}) \Big) \\
={}& 1 + 4\rho^2-8 A \rho^{4-n} + o_1(\rho^{4-n}) \\
={}& (\rho^2 + \frac{1}{4})(4 -8A \rho^{2-n} +  o_1(\rho^{2-n})).
\end{align*}
Hence, putting these facts together yields
\begin{align*}
\varphi^{-1}(1+\vert \hat{\nabla} \varphi \vert^2)  &= (1 - \frac{2}{n-2} A \rho^{2-n} + o_2(\rho^{2-n}))(4-8 A\rho^{2-n} +  o_1(\rho^{2-n})) \\
&= 4 -\frac{8(n-1)}{n-2} A \rho^{2-n} +  o_1(\rho^{2-n}).
\end{align*}
This together with Theorem \ref{thm:Green_expan}(2) leads to
\begin{equation}\label{Asymp_Fform}
F = -\frac{8(n-1)}{n-2} A \rho^{2-n} +  o_1(\rho^{2-n}) = -\frac{2}{n-2} m(\hat{g})\rho^{2-n} +  o_1(\rho^{2-n})
\end{equation}
and
$$
F^2 + 8 F =  -\frac{16}{n-2} m(\hat{g})\rho^{2-n} +  o_1(\rho^{2-n}).
$$
Also, we have
$$\partial_\rho F=2m(\hat{g})\rho^{1-n}+o(\rho^{1-n}).$$
By \eqref{Lap_F} and \eqref{Asymp_Fform}, we apply the maximum principle to show that $F \leq 0$ in $\widehat M$.

Define $\widehat M_\rho:=\widehat M \cap \{|y|<\rho\}$ for $\rho \gg 1$. Integrating \eqref{Lap_Aux-fcn} over $\widehat M_\rho$ yields
\begin{align*}
 &8 m(\hat g)|\Sp^{n-1}|+o(1)\\
 =&\frac{1}{2}\int_{\pa \widehat M_\rho}\frac{\pa}{\pa \hat \nu}(F^2+8F) \ud \sigma_{\hat g}\\
=&\int_{\widehat M_\rho} |\hat \nabla F|^2 \ud \mu_{\hat g}+2\int_{\widehat M_\rho} \varphi^{-2}|\hat \nabla \varphi|^2 \vert \mathring{\hat{\nabla}}^2\varphi \vert^2\ud \mu_{\hat g}+\int_{\widehat M_\rho} \varphi^{-1}\hat \Delta F \ud \mu_{\hat g}.
\end{align*}
Finally, letting $\rho \to \infty$ we obtain \eqref{G-M mass identity}. 
\end{proof}

\subsection{Mass and the volume gap}

It remains to prove \eqref{lbd-mass:volume gap} of Theorem \ref{Thm:G-M type mass formula}. We begin with the following lemma, which also follows from Kim-Park \cite[(2.12) in Proposition 2.4]{KimPark} using a different method:

\begin{lemma}\label{lem:key-est}
For $n \geq 5$, there holds
\begin{equation}\label{lbd:geometric}
-\int_M FG \ud \mu_g \geq \frac{n+2}{(n-2)2^{n-4}}(|\Sn|-V).
\end{equation}
\end{lemma}
\begin{proof}

Let $\Phi=\Phi(\rho)$ be a smooth function in $(0,\infty)$ satisfying
$$\bigl(\frac{\md^2 }{\md \rho^2}+\frac{n-1}{\rho} \frac{\md}{\md \rho}\bigr)\Phi=(\rho^2+\frac{1}{4})^{-n}, \qquad \lim_{\rho \searrow 0}\rho^{n-1}\frac{\md \Phi}{\md \rho}=0.$$
We let $t=\rho^2+1/4 \in [1/4,\infty)$ and abuse $\Phi(t)=\Phi(\rho(t))$. Then
\begin{equation}\label{ODE:Phi}
(4t-1) \Phi''+2n \Phi'=t^{-n}, \quad \Phi'(1/4)=\frac{4^n}{2n}.
\end{equation}
Thus, we obtain
\begin{equation}\label{eq:dPhi}
\Phi'(t)=\frac{I(\rho)}{2\rho^{n}}, \quad I(\rho)=\int_0^\rho \frac{s^{n-1}}{(\frac{1}{4}+s^2)^n} \ud s.\end{equation}
Clearly, the singularity of $\Phi'(t)$ at $t=1/4$ is removable.

Recall that $\varphi=G^{\frac{2}{n-2}}$.  By \eqref{Lap_varphi}, we know that 
$$\hat  \Delta \varphi=\frac{n}{2}(F+4), \quad |\hat\nabla \varphi |^2=(F+4)\varphi -1.$$
Now for $\Phi=\Phi(\varphi)$, we compute
\begin{equation}\label{eq:Lap_Phi}
\begin{aligned}
\hat \Delta \Phi(\varphi)=&\Phi' \hat \Delta \varphi+\Phi'' |\hat \nabla \varphi|^2\\
=&\frac{n}{2}(F+4)\Phi' +\Phi''\bigl[(F+4)\varphi -1\bigr]\\
=&(4\varphi-1) \Phi''+2n \Phi'+(\varphi \Phi''+\frac{n}{2} \Phi')F\\
=&\varphi^{-n}+K(\varphi) F,
\end{aligned}
\end{equation}
where 
$$K(t)=t \Phi''+\frac{n}{2} \Phi'=\frac{1}{4t-1}\left(t^{1-n}-\frac{n}{2}\Phi'\right).$$
by \eqref{ODE:Phi}.

Notice that $I(\infty)|\Sp^{n-1}|=|\Sn|$. Under asymptotic coordinates $y$, let $\widehat M_t:=\widehat M \cap \{|y|<t\}$ for $t \gg 1$. Integrating \eqref{eq:Lap_Phi} over $\widehat M_t$ yields
\begin{align*}
\int_{\pa B_t(0)}\Phi'\pa_t \varphi \sigma_{\hat g}=\int_{\pa B_t(0)}\pa_t \Phi(\varphi) \ud \sigma_{\hat g}=\int_{\widehat M_t} \bigl(\varphi^{-n} +K(\varphi) F\bigr)\ud \mu_{\hat g}
\end{align*}
Letting $t \to \infty$ and by \eqref{Green-fcn-y} and \eqref{eq:factor} we obtain
\begin{equation}\label{eq:volume gap}
|\Sn|-V=\int_M (-FG)H(\varphi) \ud \mu_g,
\end{equation}
where 
$$H(t)=-K(t) t^{\frac{n+2}{2}}=\frac{1}{4t-1}\left(\frac{n}{2}t^{\frac{n+2}{2}}\Phi'-t^{2-\frac{n}{2}}\right).$$

We now claim that $H(t)$ is monotone decreasing in $[1/4,\infty)$, that is,
$$H(t) \leq H(1/4)=\frac{(n-2)2^{n-4}}{n+2}.$$
Since $-FG \geq 0$ in $\widehat M$, \eqref{eq:volume gap} yields
$$|\Sn|-V\leq H(1/4)\int_{M}(-FG) \ud \mu_g=\frac{(n-2)2^{n-4}}{n+2}\int_{M}(-FG) \ud \mu_g,$$
which implies \eqref{lbd:geometric}.

It remains to prove the above claim. 

To obtain $H(1/4)$, we differentiate ODE $(\ref{ODE:Phi})$ to show
$$
(4t-1)\Phi'''+(2n+4)\Phi'' = -nt^{-n-1}.
$$
Evaluating $t = \frac{1}{4}$ yields
$$
\Phi''(1/4) = -\frac{n4^{n+1}}{2n+4}.
$$
Then we have
\begin{align*}
K(1/4) &= \frac{1}{4}\Phi''(1/4)  + \frac{n}{2}\Phi'(1/4) \\
&= -\frac{1}{4}\frac{n4^{n+1}}{2n+4} + \frac{n}{2} \frac{4^n}{2n} \\
&= -4^{n-1} \frac{n-2}{n+2},
\end{align*}
thereby,
$$
H(1/4) = -K(1/4) 4^{-\frac{n+2}{2}} = \frac{(n-2)2^{n-4}}{n+2}.
$$

Next we compute 
\begin{align*}
H'(t) ={}& \frac{1}{(4t-1)^2} \bigg[ \bigg(\frac{n(n+2)}{4} t^{\frac{n}{2}} \Phi' + \frac{n}{2} t^{1+\frac{n}{2}}\Phi'' - (2 - \frac{n}{2}) t^{1 - \frac{n}{2}} \bigg)(4t-1) - 4 \bigg(\frac{n}{2} t^{1+\frac{n}{2}} \Phi' - t^{2-\frac{n}{2}}\bigg)\bigg] \\
={}& \frac{1}{(4t-1)^2} \bigg[ \frac{n(n+2)}{4} (4t-1)t^{\frac{n}{2}} \Phi' + \frac{n}{2} t^{1+\frac{n}{2}} (t^{-n} - 2n\Phi') \\
&\qquad \qquad  - (2 - \frac{n}{2}) t^{1 - \frac{n}{2}} (4t-1)- 4 \bigg(\frac{n}{2} t^{1+\frac{n}{2}} \Phi' - t^{2-\frac{n}{2}}\bigg)\bigg] \\
={}& \frac{t^\frac{n}{2}}{(4t-1)^2} \bigg(-\frac{n(n+2)}{4} \Phi'  +(2n-4) t^{2-n} + 2 t^{1-n}\bigg) \\
={}& \frac{t^\frac{n}{2}}{2\rho^n(4t-1)^2 } \bigg(-\frac{n(n+2)}{4} I(\rho)  + 2(2n-4) \rho^n (\rho^2+\frac{1}{4})^{2-n} + 4\rho^n (\rho^2+\frac{1}{4})^{1-n}\bigg)\\
:=& \frac{t^\frac{n}{2}}{2\rho^n(4t-1)^2 } F(\rho).
\end{align*}
Since $H'$ and $F(\rho)$ have the same sign, it suffices  to compute
\begin{align*}
\frac{\md}{\md\rho} F ={}& -\frac{n(n+2)}{4} \rho^{n-1}(\rho^2 +  \frac{1}{4})^{-n} + 2n(2n-4) \rho^{n-1} (\rho^2+\frac{1}{4})^{2-n} \\
&+ 4(2-n)(2n-4) \rho^{n+1} (\rho^2+\frac{1}{4})^{1-n} + 4n \rho^{n-1} (\rho^2+\frac{1}{4})^{1-n} + 8(1-n) \rho^{n+1} (\rho^2+\frac{1}{4})^{-n} \\
={}& \bigg[-\frac{n(n+2)}{4} + 2n(2n-4)(\rho^4+ \frac{1}{2} \rho^2 +\frac{1}{16})  \\
& ~+4(2-n)(2n-4) (\rho^4 + \frac{1}{4} \rho^2 ) +4n(\rho^2+\frac{1}{4}) + 8(1-n)\rho^2\bigg]\rho^{n-1}(\rho^2 +  \frac{1}{4})^{-n}\\
={}&4(4-n)(n-2) \rho^{n+3}(\rho^2 +  \frac{1}{4})^{-n}\leq 0, \qquad \mathrm{~~when~~} n\geq 5.
\end{align*}
So, $F(\rho) \leq F(0)=0$ for $\rho \in [0,\infty)$.  This implies that $H(t)$ is monotone decreasing. 
\end{proof}

\begin{proof}[Proof of \eqref{lbd-mass:volume gap} of Theorem \ref{Thm:G-M type mass formula}.]

By \eqref{est:lbd_F} of Theorem \ref{Thm:G-M type mass formula}, we obtain
\begin{equation}\label{est:lbd-G}
F=\varphi^{-1} (1 + \vert \hat{\nabla} \varphi \vert^2) -4 \leq 0 \quad \Rightarrow \quad \varphi \geq  \frac{1}{4}.
\end{equation}
For $n\geq 5$, by \eqref{est:lbd-G} and \eqref{Lap_F} we have
$$\int_{\widehat M} \varphi^{-1}\hat \Delta F \ud \mu_{\hat g}\geq 2^{4-n} \int_{\widehat M} G^{-1}\hat \Delta F \ud \mu_{\hat g}.$$
By \eqref{eq:Lap_G-inverse}, integrating by parts together with asymptotic behavior of $G$ and $F$ yields
$$\int_{\widehat M} G^{-1}\hat \Delta F \ud \mu_{\hat g}=-\lambda_n \int_{M} G F \ud \mu_g.$$
By Lemma \ref{lem:key-est} we obtain
$$\int_{\widehat M} \varphi^{-1}\hat \Delta F \ud \mu_{\hat g}\geq n(n+2)2^{2(3-n)} (|\Sn|-V).$$
This together with \eqref{G-M mass identity} yields
$$m(\hat g) \geq \frac{n(n+2)}{2^{2n-3}|\Sp^{n-1}|} (|\Sn|-V):=\alpha_n (|\Sn|-V),$$
which is exactly \eqref{lbd-mass:volume gap}.
\end{proof}

\begin{remark}
One can give an alternate proof of the inequality \eqref{lbd-mass:volume gap} in Theorem \ref{Thm:G-M type mass formula} by combining our Theorem \ref{thm:Green_expan} with \cite[Proposition 2.4, Lemma 2.6]{KimPark}. We leave the details to interested readers. As a byproduct, we give a geometric interpretation of the normalized integral gap in \cite[(1.4)]{KimPark} under the condition \eqref{cond:Weyl-vanishing}.
\end{remark}

\section{The sharp mass-volume inequality}\label{sec:d-m_ineq}
 For a domain $\Omega \subset M$, we define the inner product on $H^1(\Omega,g)$ by
\begin{equation*}
 \mathcal Q_\Omega(u,v)
 =\int_\Omega\left(\langle\nabla u,\nabla v\rangle
                       +\lambda_n uv\right)\ud \mu_{g}
\end{equation*}
and let $ E_\Omega[u]=\mathcal Q_\Omega(u,u)$ for  $u,v \in H^1(\Omega,g)$. Take $\Omega_\varepsilon=M\setminus B_\varepsilon(\mathsf p)$ with
$\varepsilon \ll 1$. In particular, the corresponding  Yamabe energy on a spherical can $B_{\fR}\setminus B_{\varepsilon}(\mathsf p) \subset \Sn$ is
\begin{equation}\label{eq:capenergy}
  E_{\mathrm{can}}(\varepsilon)
 =|\Sp^{n-1}|\int_\varepsilon^{\fR} e_{\fR}(t)\sin^{n-1}t\ud t.
\end{equation}

\begin{lemma}\label{lem:deficitenergy}
Assume the same assumptions as in Theorem \ref{thm:diameter-mass inequality}. Then we have
\begin{equation}\label{eq:truncateddeficit}
  E_{\mathrm{can}}(\varepsilon)
       - E_{\Omega_\varepsilon}[F_{\fR}]=-e_{\fR}(\varepsilon)D(\varepsilon)
             +\int_\varepsilon^{\fR} D(t)(-e_{\fR}'(t))\ud t
\end{equation}
and 
\begin{equation}\label{eq:deficitlimit}
 \lim_{\varepsilon\searrow0}
 \left( E_{\mathrm{can}}(\varepsilon)
       - E_{\Omega_\varepsilon}[F_{\fR}]\right)
 =\int_0^\fR D(t)(-e_\fR'(t))\ud t<\infty.
\end{equation}
\end{lemma}

\begin{proof}
Notice that $F_\fR$ is Lipschitz away from $\mathsf p$; $\nabla F_{\fR}$ is
$f_{\fR}'(r)\nabla r$ where $0<r<\fR$, and zero where $r\geq \fR$, almost everywhere.
The coarea formula gives
\begin{equation}\label{eq:extendedenergy}
\begin{aligned}
  E_{\Omega_\varepsilon}(F_\fR)
 &=\int_\varepsilon^\fR e_\fR(t)\ud V_{\mathsf p}(t)
   +\lambda_n f_\fR(\fR)^2\bigl(V-V_{\mathsf p}(\fR)\bigr)\\
 &=\int_\varepsilon^\fR e_\fR(t)\ud V_{\mathsf p}(t)+e_\fR(\fR)D(\fR),
\end{aligned}
\end{equation}
where we used  $V_n(\fR)=V$ and $f_\fR'(\fR)=0$ in the second equality.
Subtracting \eqref{eq:extendedenergy} from \eqref{eq:capenergy} and
integrating by parts give
\[
 \int_\varepsilon^\fR e_\fR(t)\ud D(t)-e_\fR(\fR)D(\fR)
 =-e_\fR(\varepsilon)D(\varepsilon)
  -\int_\varepsilon^\fR e_\fR'(t)D(t)\ud t.
\]
This gives \eqref{eq:truncateddeficit}.

By Theorem \ref{thm:Green_expan}(2) and Remark \ref{rem:lowdim}, take $q=1+\lfloor\frac{n-6}{2}\rfloor$ when $n\geq 6$ and $q=0$ when $n=4,5$. By \eqref{eq:angularbound-2} we obtain
\begin{equation}\label{eq:smallvolumedeficit}
 D(t)=O(t^{\min\{n+2q+4, 3n-4\}}).
\end{equation}
On $\mathbb{S}^n$, we have
\[
 e_\fR(t)=O(t^{2-2n}),\qquad -e_\fR'(t)=O(t^{1-2n}).
\]
Consequently,
\begin{align}
 e_\fR(\varepsilon)D(\varepsilon)
       =&O(\varepsilon^{\min\{2q+6-n, n-2\}})\rightarrow0 \quad \mathrm{~~as~~}\varepsilon \to 0,
 \label{eq:boundarydeficitvanishes}\\
 -D(t)e_\fR'(t)=&O(t^{\min\{2q+5-n, n-3\}})\in L^1(0,\fR).
 \label{eq:deficitintegrability}
\end{align}
This implies \eqref{eq:deficitlimit}.
\end{proof}

For brevity, let $a_{\fR}=2^{2-n}\cot^n(\fR/2)$.
\begin{lemma}\label{lem:poleterms}
There hold $w\in H^1(M,g)$ and as $\varepsilon \to 0$,
\begin{align}
  E_{\Omega_\varepsilon}[G]
       - E_{\mathrm{can}}(\varepsilon)
   &\rightarrow(n-2)|\Sp^{n-1}| (A-a_\fR),
 \label{eq:poleenergy}\\
  \mathcal Q_{\Omega_\varepsilon}(G,w)
   &\rightarrow(n-2)|\Sp^{n-1}| (A-a_\fR).
 \label{eq:polemixed}
\end{align}
\end{lemma}

\begin{proof}
Observe that
$$f_\fR(r)=\sigma(r)+a_{\fR}+O(r^2).$$
This together with \eqref{eq:regular-remainder} yields
\begin{equation*}
 w=A-a_\fR+o(1),\qquad r|\nabla w|=o(1),\qquad r^2|\nabla^2w|=o(1)
\end{equation*}
near $\mathsf p$. Away from $\mathsf p$, $F_\fR$ is  Lipschitz and $G$ is smooth.  So,
$w\in H^1(M,g)$.

Since $L_gG=0$ on $\Omega_\varepsilon$, integrating by
parts gives
\begin{align}
  E_{\Omega_\varepsilon}[G]
 &=-\int_{\partial B_\varepsilon(\mathsf p)}G\partial_rG\ud \sigma_{g},
 \label{eq:greenboundary}\\
 \mathcal Q_{\Omega_\varepsilon}(G,w)
 &=-\int_{\partial B_\varepsilon(\mathsf p)}w\partial_rG \ud \sigma_{g}.
 \label{eq:mixedboundary}
\end{align}

Write $\sigma_{\mathsf p}(t)=\mathrm{Area}_{g}(\partial B_t(\mathsf p))$ and
$\sigma_n(t)=|\Sp^{n-1}|\sin^{n-1}t$. Since $f_\fR$ solves the radial equation \eqref{eq:Green-fcn-sphere}, we have
\[
 \bigl(\sigma_n f_{\fR} f_{\fR}'\bigr)'=\sigma_n e_\fR
\]
for all $t$ smaller than the injectivity radius.
By \eqref{eq:capenergy} and $f_\fR'(\fR)=0$, we obtain
\begin{equation}\label{eq:capboundary}
  E_{\mathrm{can}}(\varepsilon)
       =-\sigma_n(\varepsilon)f_\fR(\varepsilon)f_\fR'(\varepsilon).
\end{equation}
Near $\mathsf p$, the estimate \eqref{eq:angularbound-2} gives
\begin{equation}\label{eq:smallareadefect}
 \sigma_n(t)-\sigma_{\mathsf p}(t)=O(t^{\min\{n+2q+3, 3n-5\}}).
\end{equation}
This together with $f_\fR f_\fR'=O(t^{3-2n})$ yields
$$\bigl[f_\fR f_\fR'(\sigma_n-\sigma_{\mathsf p})\bigr](t)=O(t^{\min\{2q+6-n, n-2\}}).$$

On $\partial B_\varepsilon(\mathsf p)$, $G=f_\fR+w$. Hence
\eqref{eq:greenboundary} and \eqref{eq:capboundary} give
\begin{align}
 & E_{\Omega_\varepsilon}[G]
       - E_{\mathrm{can}}(\varepsilon)\notag\\
 =&\bigl[f_\fR f_\fR'(\sigma_n-\sigma_{\mathsf p})\bigr](\varepsilon)
   -\int_{\partial B_\varepsilon(\mathsf p)}
       \bigl(wf_\fR'+f_\fR\partial_rw+w\partial_rw\bigr)\ud \sigma_{g}.
 \label{eq:boundaryexpansion}
\end{align}
Notice that $-f_\fR'(\varepsilon)=(n-2)\varepsilon^{1-n}(1+o(1))$ and then
\[
 -\lim_{\varepsilon \searrow 0}\int_{\partial B_\varepsilon(\mathsf p)}wf_\fR'\ud \sigma_{g}
       =(n-2)|\Sp^{n-1}| (A-a_\fR).
\]
Finally, the remaining two terms are bounded by
\[
 C\varepsilon^{2-n}o(\varepsilon^{-1})\varepsilon^{n-1}=o(1),
 \qquad C \varepsilon^{n-1} o(\varepsilon^{-1})=o(1),
\]
respectively. This proves \eqref{eq:poleenergy}. Similarly,
\eqref{eq:mixedboundary} together with
$\partial_r G=f_\fR'+\partial_r w$ gives \eqref{eq:polemixed}.
\end{proof}

\begin{proof}[Proof of Theorem \ref{thm:diameter-mass inequality}.]
For $w=G-F_\fR$, a direct computation yields
\begin{equation}\label{eq:energyalgebra}
\begin{aligned}
  E_{\Omega_\varepsilon}[w]
 ={}& E_{\Omega_\varepsilon}[F_\fR]
       - E_{\Omega_\varepsilon}[G]
       +2 \mathcal Q_{\Omega_\varepsilon}(G,w)\notag\\
 ={}&\bigl( E_{\Omega_\varepsilon}[F_\fR]
                    - E_{\mathrm{can}}(\varepsilon)\bigr)
       -\bigl( E_{\Omega_\varepsilon}[G]
                    - E_{\mathrm{can}}(\varepsilon)\bigr)
       +2 \mathcal Q_{\Omega_\varepsilon}(G,w).
\end{aligned}
\end{equation}
Since $w\in H^1(M,g)$, letting
$\varepsilon\searrow0$ \eqref{eq:energyalgebra} together with Lemmas \ref{lem:deficitenergy}- \ref{lem:poleterms} gives
\[
  E[w]
 =\int_0^\fR D(t)e_\fR'(t)\ud t
       +(n-2)|\Sp^{n-1}|(A-a_\fR).
\]
By \cite[Proposition 2.1]{Viaclovsky} and Theorem \ref{thm:Green_expan}, we know that $4(n-1) A=m(\hat g)$. Finally, putting these facts together yields \eqref{eq:capidentity}.

Next we claim that  $e_\fR'(t)<0$ in $(0,\fR)$. To this end, for $t \in (0,\fR)$ we have
$$\sigma'=\frac{2-n}{2} \sigma \cot \frac{t}{2}, \quad \zeta'=\frac{n-2}{2}\zeta \tan \frac{t}{2}>0, $$
thereby,
$$f_{\fR}'=\sigma'+k_{\fR}\zeta'=\zeta'\bigl(\cot^n \frac{\fR}{2}-\cot^n  \frac{t}{2}\bigr)<0.$$
A direct computation also yields $f_{\fR}''>0$ in $(0,\fR)$. So, we obtain
$$e_{\fR}'=2f_{\fR}'(f_\fR''+\lambda_n f_\fR)<0 \quad \mathrm{~~in~~} (0,\fR).$$

By the above claim and \eqref{eq:volumedeficit}, \eqref{eq:capidentity} implies
\[
 m(\hat g)\ge (n-1)2^{4-n}\cot^n(\fR/2)\ge (n-1)2^{4-n}\cot^n(\diam(M,g)/2).
\]
 If $m(\hat g)=0$, then it forces $\fR=\pi$ and $\diam(M,g)=\pi$.
 
 \medskip
 
 Now suppose the equality of the first inequality in \eqref{eq:massdiam} holds.
Then it follows from \eqref{eq:capidentity} that  $E[w]=0$ and
$\int_0^\fR D(t)e_\fR'(t)\ud t=0$.
Consequently, we obtain $w=0$ and 
\begin{equation*}
 D(t)=0 \quad \mathrm{~~in~~} (0,\fR).
\end{equation*}
In particular, $V_{\mathsf p}(\fR)=V_n(\fR)=V$, the geodesic ball of radius $\fR$ has full
volume, and no point can have distance greater than $\fR$, as such points
would contain a nonempty open set.

By the rigidity of Bishop-Gromov volume comparison theorem, $(M, g)$ is locally isometric to the unit round sphere. Hence $M=\Sn/\Gamma$ for a finite group of
orthogonal transformations acting freely.

Let $\widetilde {\mathsf p}\in\Sn$ lift $\mathsf p$. The lift of the Green function on the universal covering $\Sn$ is
\begin{equation}\label{eq:quotientgreen}
 \widetilde G(x)=
 \sum_{\gamma\in\Gamma}\sigma(\md_{\Sn}(x,\gamma\widetilde {\mathsf p})),
\end{equation}
see for example \cite[Lemma 2.4]{EscSch}.
Since $w=0$, we  have 
\[
G(y)=F_\fR(y)=\sigma(\md_{g}(\mathsf p, y))+k_\fR \zeta(\md_{g}(\mathsf p, y)).
\]
Then for any point $x$ in a small ball around $\widetilde {\mathsf p} \in \Sn$, we set 
\[
t:=\md_{g}(\pi(x), \mathsf p)=\md_{\mathbb{S}^n}(x, \widetilde{\mathsf p}).
\]
Hence,
\[
\widetilde{G}(x)=G(\pi(x))=\sigma(t)+k_\fR \zeta(t).
\]
Since $\sigma(\md_{\Sn}(x,-\widetilde{\mathsf p}))=\sigma(\pi-t)=(2 \cos (t / 2))^{2-n}=\zeta(t)$, we have
\begin{equation}\label{eq:poleidentity}
 \sum_{\gamma\in\Gamma}\sigma(\md_{\Sn}(x,\gamma\widetilde {\mathsf p}))
 =\sigma(\md_{\Sn}(x,\widetilde p))+k_\fR\sigma(\md_{\Sn}(x,-\widetilde {\mathsf p})).
\end{equation}
If $k_\fR=0$, then the orbit $\{\gamma \widetilde {\mathsf p}\}$ has only $\widetilde {\mathsf p}$, so $\Gamma=\{I\}$, which implies that $(M, g)$ is isometric to  $(\Sn,g_{\Sn})$.
If $k_\fR>0$, then the orbit is $\{\widetilde {\mathsf p},-\widetilde {\mathsf p}\}$, its two
coefficients are one, and therefore $|\Gamma|=2$.
In this case, $(M, g)$ is isometric to the round $\RP^n$ with
$\fR=\diam(M,g)=\pi/2$.
\end{proof}

\begin{proof}[Proof of Corollary \ref{Cor:Gap theorem}]
Let $z=\cot(\fR/2)$. Then
\begin{align*}
 |\Sn|-V
 &=|\Sp^{n-1}|\int_R^\pi\sin^{n-1}s \ud s\notag\\
 &=2^n|\Sp^{n-1}|\int_0^z\frac{\tau^{n-1}}{(1+\tau^2)^n}\ud \tau
 \le\frac{2^n|\Sp^{n-1}|}{n}z^n
 \le K_n m(\hat g),
\end{align*}
where
\[
 K_n=\frac{2^{2n-3}|\Sp^{n-1}|}{n}.
\]
Let $\eta_n\in(0,1)$ be the constant in \cite[Theorem 4.4]{Carron}, which also traces back to Anderson \cite[Theorem 1.2]{Anderson}: If $(M^n, g)$ is an Einstein manifold with $\Ric=(n-1)g$, and $\Vol(M,g)>(1-\ve_n) |\Sn|$, then $(M,g)$ is  isometric to $(\Sn, g_{\Sn})$. Though Anderson did not state this clearly, one can prove this from Anderson's result, as explained in \cite{Carron}.
Hence, by choosing 
\begin{equation*}
 \delta_n=\frac{\ve_n |\Sn|}{K_n},
\end{equation*}
we conclude that for $m(\hat g)<\delta_n$, $(M,g)$ is isometric to $(\Sn, g_{\Sn})$ and $m(\hat g)=0$. 
\end{proof}

Finally, we prove the blow-up phenomenon  for the masses of volume collapsing sequence of Einstein metrics.

\begin{proof}[Proof of Corollary \ref{Cor:mass_blowup}.]
Let $\mathfrak R_i\in(0,\pi]$ be the volume radius of  $(M,g_i)$ such that
$$V_{g_i}=|\Sp^{n-1}|\int_0^{\mathfrak R_i}\sin^{n-1} s\ud s.$$
Since $V_{g_i}\to 0$ as $i\to\infty$, we have $\mathfrak R_i\to 0$, whence, $\cot\frac{\mathfrak R_i}{2}\to\infty $. Now by Theorem \ref{thm:diameter-mass inequality} we have
$$m(\widehat{g_i})\ge (n-1)2^{4-n}\cot^n(\fR_i/2)\to\infty.$$
Also, since 
$$\lim_{r\to 0}\cot^n\frac{r}{2}\int_0^r\sin^{n-1}s \ud s=\frac{2^n}{n},$$
we have 
$$\liminf_{i\to\infty} V_{g_i}m(\widehat{g_i})\geq \frac{16(n-1)}{n}|\Sp^{n-1}|.$$
This completes the proof.
\end{proof}

\bibliographystyle{amsplain}
\bibliography{references}
\end{document}